\documentclass[a4paper,11pt]{article}       

\newcounter{CounterTodo}
\usepackage{xcolor}

\usepackage{authblk}
\usepackage{amsmath}
\usepackage{amssymb} 
\usepackage{amsthm}
\usepackage{complexity}
\usepackage{enumerate}
\usepackage{fullpage}
\usepackage{hyperref}
\usepackage{lineno}
\usepackage{setspace}
\usepackage{tikz}

\newtheorem{theorem}{Theorem}[section]
\newtheorem{proposition}[theorem]{Proposition}
\newtheorem{lemma}[theorem]{Lemma}
\newtheorem{corollary}[theorem]{Corollary}

\newcommand{\exclusive}[1]{\overset{\leftarrow}{#1}}
\newcommand{\shared}[1]{\overset{\rightarrow}{#1}}

\begin{document}

\onehalfspace

\title{Computing the hull and interval numbers \\ in the weakly toll convexity}

\author[1]{Mitre C. Dourado\thanks{Partially supported by CNPq, Brazil, Grant number 305404/2020-2 and FAPERJ, Brazil, Grant number E-26/211.753/2021. E-mail: mitre@ic.ufrj.br.}}
\author[2]{Marisa Gutierrez\thanks{E-mail: marisa@mate.unlp.edu.ar.}}
\author[3]{\\ Fábio Protti\thanks{Partially supported by Conselho Nacional de Desenvolvimento Cient\'{\i}fico e Tecnol\'{o}gico (304117/2019-6) and FAPERJ (201.083/2021), Brazil. E-mail: fabio@ic.uff.br.}}
\author[2]{Silvia Tondato\thanks{E-mail: tondato@mate.unlp.edu.ar.}}

\affil[1]{Instituto de Computação, Universidade Federal do Rio de Janeiro, Brazil.}
\affil[2]{Facultad de Ciencias Exactas, Universidad Nacional de La Plata, Argentina.}
\affil[3]{Instituto de Computa\c c\~ao, Universidade Federal Fluminense, Brazil.}

\maketitle

\begin{abstract}
A walk $u_0u_1 \ldots u_{k-1}u_k$ of a graph $G$ is a \textit{weakly toll walk} if $u_0u_k \not\in E(G)$, $u_0u_i \in E(G)$ implies $u_i = u_1$, and $u_ju_k\in E(G)$ implies $u_j=u_{k-1}$. The {\em weakly toll interval} of a set $S \subseteq V(G)$, denoted by $I(S)$, is formed by $S$ and the vertices belonging to some weakly toll walk between two vertices of $S$. Set $S$ is {\it weakly toll convex} if $I(S) = S$.
The {\em weakly toll convex hull} of $S$, denote by $H(S)$, is the minimum weakly toll convex set containing $S$. The {\em weakly toll interval number} of $G$ is the minimum cardinality of a set $S \subseteq V(G)$ such that $I(S) = V(G)$; and the {\em weakly toll hull number} of $G$ is the minimum cardinality of a set $S \subseteq V(G)$ such that $H(S) = V(G)$. In this work, we show how to compute the weakly toll interval and the weakly toll hull numbers of a graph in polynomial time. In contrast, we show that determining the weakly toll convexity number of a graph $G$ (the size of a maximum weakly toll convex set distinct from $V(G)$) is \NP-hard. 

\medskip

	
\end{abstract}

\section{Introduction} \label{sec:int}

A family ${\cal C}$ of subsets of a finite set $X$ is a {\em convexity on $X$} if $\varnothing, X \in {\cal C}$ and ${\cal C}$ is closed under intersections~\cite{van-de-vel}. The members of $\cal{C}$ are called \textit{convex sets}. A \textit{graph convexity} is a convexity for which $X=V(G)$ for a graph $G$. There are many examples in the literature where a graph convexity is defined over a path system; the {\em monophonic convexity}~\cite{dourado-et-al,duchet} consists of all the monophonically convex sets of $V(G)$ (a set $S$ is {\em monophonically convex} if and only if every {\em induced} path between two vertices of $S$ lies entirely in the subgraph induced by $S$). Likewise, the {\em geodesic convexity}~\cite{pelayo}, the {\em $m^3$-convexity}~\cite{dragan-et-al}, the {\em toll convexity}~\cite{alcon-et-al}, and the {\em weakly toll convexity}~\cite{gutierrez-tondato} are defined over shortest paths, induced paths of length at least three, tolled walks, and weakly toll walks, respectively.

In this work, we consider the weakly toll convexity, a variation of the toll convexity. 
Given distinct and nonadjacent vertices $u$ and $v$ of a graph $G$, a {\em weakly toll $(u,v)$-walk} is a walk $u_0u_1\ldots u_{k-1}u_k$ such that $u=u_0$, $v=u_k$, and the following conditions hold:

\begin{itemize}
	\item $u_0u_i\in E(G)$ implies $u_i=u_1$;
	\item $u_iu_k\in E(G)$ implies $u_i=u_{k-1}$.
\end{itemize}

\noindent In other words, $u_0$ (resp., $u_k$) has only one neighbor in the walk: vertex $u_1$ (resp., $u_{k-1}$), which can appear more than once in the walk.

The concept of weakly toll walk is a relaxation of the concept of {\em tolled walk}~\cite{alcon-et-al}, defined to capture the structure of the convex geometry associated with an interval graph (for the concept of {\em convex geometry}, we refer the reader to~\cite{farber-jamison}). Analogously, weakly toll walks are used to characterize proper interval graphs as convex geometries (see~\cite{gutierrez-tondato}). In addition to interval and proper interval graphs, other important classes of graphs have been characterized as convex geometries. Chordal and Ptolemaic graphs are convex geometries with respect to the monophonic and the geodesic convexities, respectively~\cite{farber-jamison}; and weak polarizable graphs~\cite{olariu} are convex geometries with respect to the $m^3$-convexity~\cite{dragan-et-al}. Other classes of graphs that have been characterized as convex geometries are forests, cographs, bipartite graphs, and planar graphs~\cite{araujo-sampaio}. 

The {\em weakly toll interval} of a set $S \subseteq V(G)$, denoted by $I(S)$, is formed by $S$ and the vertices belonging to some weakly toll walk between two vertices of $S$. We say that $S$ is a {\em weakly toll interval set} if $I(S) = V(G)$, and that $S$ is {\it weakly toll convex} if $S = I(S)$. The {\em weakly toll convex hull} of $S$, denoted by $H(S)$, is the minimum weakly toll convex set containing $S$. We say that $S$ is a {\em weakly toll hull set} if $H(S) = V(G)$. The {\em weakly toll interval number} of $G$ is the minimum cardinality of a weakly toll interval set of $G$; the {\em weakly toll hull number} of $G$ is the minimum cardinality of a weakly toll hull set of $G$; and the {\em weakly toll convexity number} of $G$, denoted by ${\mathit wtc}(G)$, is the size of a maximum weakly toll convex set of $G$ distinct from $V(G)$.

In~\cite{gutierrez-tondato}, some invariants associated with the weakly toll convexity are studied; namely, the weakly interval number and the weakly hull number of trees and interval graphs are determined. In this work, we study the weakly toll hull, weakly toll interval, and weakly toll convexity numbers of general graphs. We remark that there are many studies on the invariants mentioned above, considering other types of graph convexities. The interested reader is referred, for instance, to~\cite{dourado-et-al,pelayo}.
We draw attention to the fact that in many cases these three parameters are \NP-complete even for special graph classes, as is the case for the hull number for partial cube graphs~\cite{2016AK} and the interval number for chordal bipartite graphs~\cite{2010DPSR} in the geodesic convexity. Regarding to the toll convexity, we cite~\cite{2022Dourado,2022Dravec}, where the hull and interval numbers have been investigated.

The remainder of this work is organized as follows. Section 2 contains all the necessary background. Section 3 presents some facts and complexity aspects on weakly toll convex sets and extreme vertices. In Section 4, we prove that deciding if the weakly toll convexity number of a graph $G$ is at least $k$ is \NP-complete, even for prime graphs ($G$ is a {\em prime graph} if it contains no clique separator). In Section 4, we show how to determine the weakly toll hull number of a graph; as in Section 3, concepts on decompositions by clique separators are used as a tool for this purpose. Section 5 is devoted to the determination of the weakly toll interval number. Finally, Section 6 contains our conclusions.

\section{Preliminaries}

We consider finite, simple and undirected graphs. Let $G$ be a graph. For a set $S\subseteq V(G)$, we denote by $G[S]$ the subgraph of $G$ induced by $S$. A vertex $v \in V(G)$ is a {\em simplicial vertex} if $N(v)$ is a clique. We say that $u$ and $v$ are ({\em true}) {\em twins} if $N[u] = N[v]$.
Denote by $T(u)$ the set of vertices that are twins to $u$, and $T[u] = T(u) \cup \{u\}$.
Note that the vertices of $G$ can be partitioned into sets $T_1, \ldots, T_k$ such that for $i \in \{1, \ldots, k\}$, $T_i = T[u]$ for some $u \in V(G)$. We call each set of this partition a {\em class of twins} of $G$ and recall that this partition can be obtained in linear time using modular decomposition~\cite{CoHa1994,McSp1994}.
Given a set $S$, we denote by $\hat{S}$ the subset of $S$ containing exactly one vertex of each class of twins having vertices in $S$.

A path (resp., walk) between two vertices $u,v \in V(G)$ is called a {\em $(u,v)$-path} (resp., {\em $(u,v)$-walk}). Let $P=u_0u_1\ldots u_{k-1}u_k$ be a $(u_0,u_k)$-path.
Given a set $S \subseteq V(G)$, a path from $u$ to $S$ is a path from $u$ to any vertex of $S$.

A vertex $x$ of a weakly toll convex set $S\subseteq V(G)$ is called a {\em weakly toll extreme vertex} of $S$ if $S\setminus\{x\}$ is also a weakly toll convex set. The set of weakly toll extreme vertices of $S$ is denoted by ${\mathit ext}(S)$. The set of weakly toll extreme vertices of $V(G)$ is denoted by ${\mathit ext}(G)$. In general, if ${\cal C}$ is a convexity on $V(G)$ and $S\in {\cal C}$, a vertex $x\in S$ is an extreme vertex of $S$ if $S\backslash\{x\}$ is also convex. The convexity ${\cal C}$ is a \textit{convex geometry} if it satisfies the \textit{Minkowski-Krein-Milman} property~\cite{krein-milman}: {\em Every convex set is the convex hull of its extreme vertices}. 

For convenience, when using the interval and hull operators for a given set, sometimes it will be useful to write the vertices of the set. For instance, we can write $I(u,v)$ standing for $I(\{u,v\})$.

We need some additional following definitions, that can be found in~\cite{leimer1993}. We say that $S\subset V(G)$ {\em separates} $u,v\in V(G)$ if $u$ and $v$ are connected by a path in $G$ but not in $G-S$. In addition, $S$ is a {\em separator} of $G$ if $S$ separates some pair of vertices of $G$, and $S$ is a {\em clique separator} of $G$ if $S$ is a separator of $G$ and a clique. A graph $G$ is {\em reducible} if it contains a clique separator, otherwise it is {\em prime}. A {\em maximal prime subgraph} of $G$, or {\em mp-subgraph} of $G$, is a maximal induced subgraph of $G$ that is prime. Given an mp-subgraph $M$ of $G$, we denote by $\shared{M}$ the set formed by the vertices of $V(M)$ that belong to at least two mp-subgraphs of $G$; and $\exclusive{M} = V(M) \setminus \shared{M}$.

An mp-subgraph $M_i$ of a reducible graph $G$ is {\em extremal} if there is an mp-subgraph $M_j\neq M_i$ for which the following property is valid: for every mp-subgraph $M_k \neq M_i$, it holds that $V(M_i)\cap V(M_k) \subseteq V(M_i)\cap V(M_j)$. In this case, note that $\shared{M_i} = V(M_i)\cap V(M_j)$.

\begin{proposition}
Let $G$ be a graph with order $n$, size $m$ and maximum degree $\Delta$.
The following hold for distinct vertices $u,v$ and $w$ of $G$ if $uw \notin E(G)$.

\begin{enumerate}
\item Then, $v$ lies in a weakly toll $(u,w)$-walk if and only if there exist $v_u \in N(u)$ and $v_w \in N(w)$ such that the graph $G-S(v_u,v_w)$ contains a connected component $C$ with $\{v_u,v_w,v\}\subseteq V(C)$, where
$$S(v_u,v_w) = (N[u] \setminus\{v_u\}) \cup (N[w] \setminus\{v_w\}).$$ \label{prop1}

\item Let $u,v,w$ be three distinct vertices of a graph $G$ such that $uv\notin E(G)$. Then, deciding whether $v$ lies in a weakly toll $(u,w)$-walk can be done in $O(\Delta^2(m+n))$ time. \label{prop2}
\end{enumerate}
\end{proposition}

\begin{proof}
$(\ref{prop1})$ Suppose that $v$ lies in a weakly toll walk $u,v_1,v_2,\ldots,v_{k-1},v_k,w$. Let $v_u = v_1$ and $v_w = v_k$. Then $v_u\in N(u)$ and $v_w \in N(w)$. In addition, $v_1,v_2,\ldots,v_{k-1},v_k$ is a walk that avoids all vertices in $N[u]$ (except $v_u=v_1$) and all vertices in $N[v]$ (except $v_w=v_k$). This implies that $v_u$, $v_w$ and $v$ lie in a same connected component of $G-S(v_u,v_w)$.

Conversely, suppose that there is a pair of (not necessarily distinct) vertices $v_u\in N(u)$ and $v_w\in N(w)$ such that $v_u$, $v_w$ and $v$ lie in a same connected component $C$ of $G-S(v_u,v_w)$. Then, there is a walk $W=v_1,v_2,\ldots,v_{k-1},v_k$ in $C$ such that $v_1=v_u$, $v_k=v_w$, and $W$ contains $v$. (Possibly, $v_u$ and $v_w$ occur more than once in $W$.) In addition, $W$ avoids all vertices in $N[u]$ (except $v_u=v_1$) and all vertices in $N[v]$ (except $v_w=v_k$). This implies that $W'=u,v_1,v_2,\ldots,v_{k-1},v_k,w$ is a weakly toll walk containing $v$.

\bigskip\noindent $(\ref{prop2})$ By Proposition~\ref{prop1}, deciding whether $v$ lies in a weakly toll $(u,w)$-walk amounts to checking whether there is pair $v_u, v_w$ of neighbors of $u$ and $w$, respectively, such that $v_u$, $v_w$ and $v$ lie in a same connected component $C$ of $G-S(v_u,v_w)$. Since there are at most $\Delta^2$ possible pairs $v_u,v_w$, and for each pair testing whether $v_u,v_w,v$ lie in a component of $G-S(v_u,v_w)$ can be done in $O(n+m)$ time, the result does follow.
\end{proof}

From the above proposition and the definitions given in the previous section, we can easily derive the following fact.

\begin{corollary}~\label{cor:complexityIH}
Given a set of vertices $S$ of a graph $G$ with order $n$, size $m$ and maximum degree $\Delta$, $I(S)$ and $H(S)$ can be computed in $O(\Delta^2 |S|^2 (n - |S|)(n+m))$ and $O(\Delta^2 n^4(n+m))$ steps, respectively.
\end{corollary}

\section{Weakly toll convexity number} \label{sec:convnumber}

In this section, we deal with the problem of determining ${\mathit wtc}(G)$ for a graph $G$. We start by showing that the {\sc Clique} Problem is \NP-complete even when restricted to prime graphs.

\begin{theorem}\label{thm:clique-prime}
	Let $G$ be a prime graph and $k$ a positive integer. Then the problem of deciding whether $G$ contains a clique of size at least $k$ is \NP-complete.
\end{theorem}

\begin{proof}
The problem is clearly in \NP. The hardness proof is a reduction from the {\sc Clique} Problem~\cite{garey-johnson}. Let $(G,k)$ be an input of the {\sc Clique} Problem. Set $k'=k$ and construct a graph $G'$ as follows:
	$$V(G') = V(G)\cup \{x_{uv} \mid u,v\in V(G), uv \notin E(G)\}$$
	$$\text{and}$$
	$$E(G')=E(G)\cup\{u\,x_{uv}, v\,x_{uv}\mid u,v\in V(G), uv\notin E(G)\}.$$
	
We show that $G'$ is a prime graph. Suppose not, and let $C$ be a clique of $G'$ that separates $a,b\in V(G')$. If $a\notin V(G)$ then $a=x_{uv}$ for some nonadjacent vertices $u,v\in V(G)$, and there is a connected component $G_a$ of $G'-C$ containing $a$ and a neighbor $a'\in\{u,v\}$ of $a$, because at least one of $u,v$ is not in $C$. Similarly, if $b\notin V(G)$, then there is a connected component $G_b\neq G_a$ of $G'-C$ containing a neighbor $b'\in V(G)$ of $b$. Thus, in any case, $C$ separates two vertices $a',b'\in V(G)$. In addition, $a'b'\notin E(G)$, and this implies that $x_{a'b'}\in V(G')$. Since $x_{a'b'}$ is adjacent only to $a'$ and $b'$, it follows that $x_{a'b'}\notin C$ and $G'-C$ contains the path $a'\,x_{a'b'}\,b'$, a contradiction. Therefore, $G'$ is a prime graph.
	
To conclude the proof, we have that $G'$ contains a clique of size at least $k'$ if and only if $G$ contains a clique of size at least $k$ because we can assume that $|V(G)| \ge 2$ and the maximum clique containing a vertex of $V(G') \setminus V(G)$ has 2 vertices.
\end{proof}

The next theorem shows that the problem of deciding if the weakly toll convexity number of a graph $G$ is at least $k$ is \NP-complete even when $G$ is a prime graph.

\begin{corollary} \label{cor:convnumber}
Let $G$ be a prime graph and $k$ a positive integer. Then, the problem of deciding whether ${\mathit wtc}(G)\geq k$ is \NP-complete.
\end{corollary}

\begin{proof}
Given a set $S\subseteq V(G)$ with $|S|\geq k$, by Corollary~\ref{cor:complexityIH}, we can check if $S$ is weakly toll convex in polynomial time. Thus the problem is in \NP.

By Corollary~\ref{cor:prime-not-clique-1}, if a graph $G$ is prime but not complete, then every weakly toll convex set different of $V(G)$ is a clique. Therefore, the hardness proof follows from the fact that the {\sc Clique} problem restricted to prime graphs is \NP-complete (Theorem~\ref{thm:clique-prime}).
\end{proof}

\section{Weakly toll interval number} \label{sec:interval}

In this section, we show how to compute ${\mathit wtn}(G)$ for any graph $G$ in polynomial time.

\begin{theorem}~\label{thm:wtn}
Let $G$ be a non-complete connected graph and let $u,v \in V(G)$ satisfying $|I(u,v)| \ge |I(w,z)|$ for any $w,z \in V(G)$. The following sentences hold.

\begin{enumerate}[$(i)$]
	\item If $u,v \notin ext(G)$, then $wtn(G) \le 8$. \label{ite:noextreme}

	\item If $u \in ext(G)$ and $v \notin ext(G)$, then $|T[u]| + 1 \le wtn(G) \le |T[u]| + 5$. \label{ite:1extreme}

	\item If $u,v \in ext(G)$, then $|T[u] \cup T[v]| \le wtn(G) \le |T[u] \cup T[v]| + 2$. \label{ite:2extremes}

	\item If $w \in ext(G)$, then $w \in T[u] \cup T[v]$. \label{ite:2classes}
\end{enumerate}
	
\end{theorem}

\begin{proof}
$(\ref{ite:noextreme})$ to~$(\ref{ite:2extremes})$
Since $G$ is not complete, the choice of the pair $\{u,v\}$ implies that $uv \notin E(G)$.
Note that exactly one of~$(\ref{ite:noextreme})$ to~$(\ref{ite:2extremes})$ holds for $G$ depending on how many of $u$ and $v$ are weakly toll extreme vertices. If $I(u,v) = V(G)$, then $wtn(G) = 2$, $|T[u]| = 1$ and $|T[v]| = 1$, which means that the result follows for any of the items~$(\ref{ite:noextreme})$ to~$(\ref{ite:2extremes})$ that fits to $G$.
Hence, we can assume that $V(G) \setminus I(u,v) \ne \varnothing$. Denote by $Y$ the subset of $V(G) \setminus I(u,v)$ formed by the vertices containing neighbors in $I(u,v)$, and write $X = V(G) \setminus (I(u,v) \cup Y)$. Since $G$ is connected, there is $y \in Y$. See Figure~\ref{fig1}.

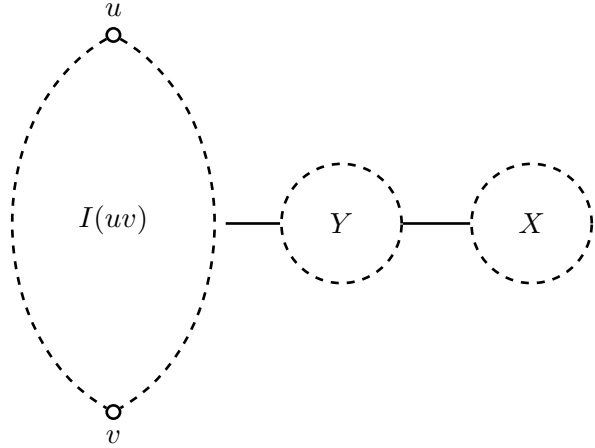
\begin{figure}[ht]
	\begin{center}

\begin{tikzpicture}[scale=1]

\pgfsetlinewidth{1pt}

\tikzset{vertex/.style={circle,  draw, minimum size=5pt, inner sep=0pt}}
\tikzset{texto/.style={circle,  minimum size=5pt, inner sep=0pt}}
\tikzset{set/.style={circle,  draw, minimum size=45pt, inner sep=5pt}}

\def\h{0}
\def\v{0}

\node [vertex] (u) at (\h , \v ) [label=above:$u$]{};
\node [vertex] (v) at (\h, \v-5) [label=below:$v$]{};
\draw [dashed] (u) to [bend left = 60] node (uv) {} (v);
\draw [dashed] (u) to [bend left = -60]   (v);
\node [texto] (I) at (\h, \v-2 ) [label=below:$I(u  v)$]{};

\node[dashed] [set] (Y) at (\h+3 , \v-2.5 ) {$Y$};
\node[dashed] [set] (X) at (\h+5.5 , \v-2.5 ) {$X$};

\draw (X) to (Y);
\draw (Y) to (uv);

\end{tikzpicture}

\end{center}
\caption{Sets $I(u,v)$, $Y$ and $X$.}
\label{fig1}
\end{figure}

We claim that $y \in N(u) \cup N(v)$. Supposing the contrary, let $y' \in I(u,v)$ be such that $yy' \in E(G)$ and let $W$ be a weakly toll $(u,v)$-walk containing $y'$.
We can see $W$ as the concatenation of two subwalks $W_u$ and $W_v$, where $W_u$ is a $(u,y')$-walk and $W_v$ is a $(v,y')$-walk. Since $W_u y W_v$ is a weakly toll $(u,v)$-walk containing $y$, we have a contradiction proving that the claim does hold.
	
We know that no vertex of $Y$ belongs to $N(u) \cap N(v)$, because this was true for some vertex $y$, then $uyv$ would be a weakly toll $(u,v)$-walk containing $y$ and $y$ would belong to $I(u,v)$. Hence, we can write $Y = Y_u \cup Y_v$ such that $Y_u \cap Y_v = \emptyset$ where $Y_u \subseteq N(u)$ and $Y_v \subseteq N(v)$.

Now, we show that $X = \varnothing$. Then, suppose that there is $x \in X$. If $Y_u \cup \{u\}$ separates $x$ from $v$, then, by the above claim, there is an $(x,u)$-path contains exactly one vertex of $Y_u$ and none of $N[v]$. Such path can be used to show that $I(u,v) \subset I(x,v)$, which is not possible by the choice of $\{u,v\}$. Therefore, by symmetry, there are paths $P_u$ and $P_v$ from $x$ to $u$ and from $x$ to $v$, respectively, such that $P_u$ contains exactly one vertex of $Y_u$ and none of $N[v]$, while $P_v$ contains exactly one vertex of $Y_v$ and none of $N[u]$. Notice that these two paths form a weakly toll $(u,v)$-walk containing $x$, which is not possible. Hence, $X = \varnothing$ and $V(G) \setminus I(u,v) = Y$. Write $Z_u = Y_u \setminus T[u]$. See Figure~\ref{fig2}.

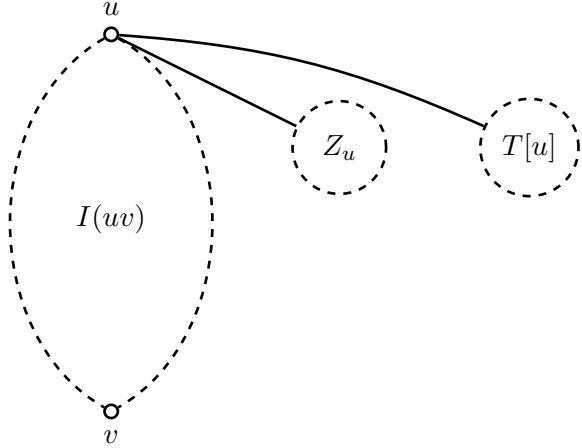
\begin{figure}[ht]
\begin{center}

\begin{tikzpicture}[scale=1]

\pgfsetlinewidth{1pt}

\tikzset{vertex/.style={circle,  draw, minimum size=5pt, inner sep=0pt}}
\tikzset{texto/.style={circle,  minimum size=5pt, inner sep=0pt}}
\tikzset{set/.style={circle,  draw, minimum size=35pt, inner sep=5pt}}

\def\h{0}
\def\v{0}

\node [vertex] (u) at (\h , \v ) [label=above:$u$]{};
\node [vertex] (v) at (\h, \v-5) [label=below:$v$]{};
\draw [dashed] (u) to [bend left = 60] node (uv) {} (v);
\draw [dashed] (u) to [bend left = -60]   (v);
\node [texto] (I) at (\h, \v-2 ) [label=below:$I(u  v)$]{};

\node[dashed] [set] (Zu) at (\h+3 , \v-1.5 ) {$Z_u$};
\node[dashed] [set] (Tu) at (\h+5.5 , \v-1.5 ) {$T[u]$};
\draw (u) to (Zu);
\draw (u) to [bend left = 10] (Tu);

\end{tikzpicture}

\end{center}
\caption{Sets $I(u,v)$, $Z_u$ and $T[u]$.}
\label{fig2}
\end{figure}

We claim that every $u' \in Z_u$ has a neighbor in $I(u,v) \setminus N[u]$. Suppose by contradiction that $u'$ has no neighbors in $I(u,v) \setminus N[u]$. We will show that $I(u',v)$ has more vertices than $I(u,v)$. It suffices to show that for every $w \in I(u,v)$, there is a weakly toll $(u',v)$-walk containing $w$.
Let $W$ be a weakly toll $(u,v)$-walk containing $w$, let $u^*$ be the only neighbor of $u$ in $W$, and denote by $W'$ the walk obtained from $W$ by deleting $w$. If $u'u^* \in E(G)$, then $u'W'$ is a weakly toll $(u',v)$-walk containing $w$. If $u'u^* \not\in E(G)$, then $u'W$ is a weakly toll $(u',v)$-walk containing $w$. Since $u' \in I(u',v) \setminus I(u,v)$, we have that $I(u',v)$ has more vertices than $I(u,v)$, which is a contradiction and the claim does hold.

Denote by $R_u$ the subset of $I(u,v) \setminus N[u]$ formed by the vertices having neighbors in $Z_u$. By the above claim, $R_u \ne \varnothing$. We have that $Q_u = N(u) \cap I(u,v)$ separates the vertices of $R_u$ from $v$, because otherwise some vertex of $Z_u$ would belong to $I(u,v)$. Now, let $r \in R_u$. See Figure~\ref{fig3}.

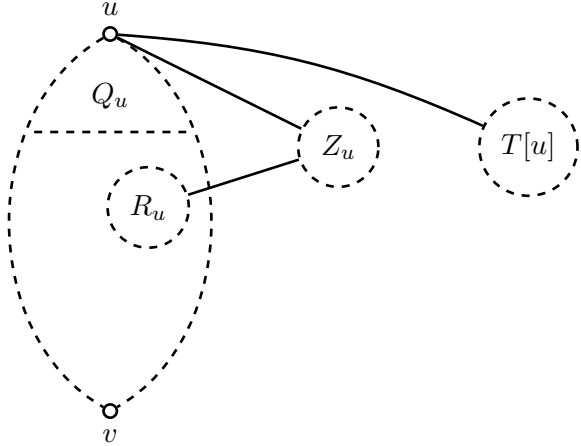
\begin{figure}[ht]
\begin{center}

\begin{tikzpicture}[scale=1]

\pgfsetlinewidth{1pt}

\tikzset{vertex/.style={circle,  draw, minimum size=5pt, inner sep=0pt}}
\tikzset{texto/.style={circle,  minimum size=5pt, inner sep=0pt}}
\tikzset{set/.style={circle,  draw, minimum size=30pt, inner sep=5pt}}

\def\h{0}
\def\v{0}

\node [vertex] (u) at (\h , \v ) [label=above:$u$]{};
\node [vertex] (v) at (\h, \v-5) [label=below:$v$]{};
\draw [dashed] (u) to [bend left = 60] node (uv) {} (v);
\draw [dashed] (u) to [bend left = -60]   (v);

\node[dashed] [set] (Zu) at (\h+3 , \v-1.5 ) {$Z_u$};
\node[dashed] [set] (Tu) at (\h+5.5 , \v-1.5 ) {$T[u]$};
\draw (u) to (Zu);
\draw (u) to [bend left = 10] (Tu);

\node[dashed] [set] (Ru) at (\h+0.5 , \v-2.3 ) {$R_u$};
\draw (Zu) to (Ru);

\draw [dashed] (-1,-1.3) to (1,-1.3);
\node [texto] (I) at (\h, \v-0.4 ) [label=below:$Q_u$]{};

\end{tikzpicture}

\end{center}
\caption{Sets $Q_u$, $R_u$, $Z_u$ and $T[u]$.}
\label{fig3}
\end{figure}

We claim that $Z_u \subset I(r,u,v)$. Let $z \in Z_u$. For the case where $rz \in E(G)$, it holds that $z \in I(r,u)$. Then, we can assume that $rz \not\in E(G)$. Let $G'$ be the subgraph of $G$ induced by $I(u,v)$. Let $P_r$ be a minimum path of $G'$ from $r$ to $Q_u$. Denote by $q$ the only vertex of $P_r$ belonging to $Q_u$. If $qv \in E(G)$, then let $P_v = vq$. Otherwise, let $P_v$ be a minimum path of $G'$ from $v$ to $Q_u$. Paths $P_r$ and $P_v$ exist because $G'$ is connected. Let $q'$ be the only vertex of $P_v$ belonging to $Q_u$, which can be equal to $q$. Furthermore, since $Q_u$ separates $R_u$ from $v$, the intersection of $P_r$ and $P_v$ has at most one vertex. Now, the walk $P_ruzuP_v$ is a weakly toll $(r,v)$-walk containing $z$, completing the proof of this claim.

Finally, define $R_v$ relative to $v$ as $R_u$ has been defined to $u$, and note that every vertex of $G$ belongs to exactly one of the sets $I(u,v), T(u), T(v), Z_u$ and $Z_v$. We consider three cases depending on whether the vertices of $T(u)$ and $T(v)$ are weakly toll extreme or not.

The first possibility is when no vertex of $T[u] \cup T[v]$ is extreme. We will show eight vertices, not necessarily different, such that every vertex of $G$ belongs to the interval of a subset of these vertices. Since the vertices of $T[u] \cup T[v]$ are not extreme, there are two vertices $w_1,w_2$ such that $T[u] \subset I(w_1,w_2)$, and there are two vertices $w_3,w_4$ such that $T[v] \subset I(w_3,w_4)$. We know that there is $r_u \in R_u$ such that $Z_u \subset I(u,v,r_u)$ and that there is $r_v \in R_v$ such that $Z_v \subset I(u,v,r_v)$. Therefore, $\{u,v, w_1, \ldots, w_4, r_u,r_v\}$ is a weakly toll interval set of $G$ with at most 8 vertices. Therefore,~$(\ref{ite:noextreme})$ does hold.
	
The second case is when exactly one of the sets $T[u]$ and $T[v]$ is formed by weakly toll extreme vertices. By symmetry, we can assume that $T[u] \subseteq ext(G)$ and that $T[v] \cap ext(G) = \varnothing$. Using the vertices of the previous case, we can conclude that $T[u] \cup \{v, w_1, w_2, r_u, r_v\}$ is a weakly toll interval set of $G$ with at least $|T[u]| + 1$ vertices and at most $|T[u]| + 5$ vertices, which means that~$(\ref{ite:1extreme})$ holds.
	
In the last case, i.e., when both $T[u]$ and $T[v]$ are formed by weakly toll extreme vertices, we have that $T[u] \cup T[v] \cup \{r_u, r_v\}$ is an interval set of $G$ with at least $|T[u] \cup T[v]|$ vertices and at most $|T[u] \cup T[v]| + 2$ vertices. Hence,~$(\ref{ite:2extremes})$ does hold.

\bigskip \noindent $(\ref{ite:2classes})$ Consider that partition $I(u,v), T(u), T(v), Z_u$ and $Z_v$ of $V(G)$ constructed above. We have shown that $Z_u \subset I(u,v,r_u)$ and $Z_v \subset I(u,v,r_v)$, which means that the vertices of $(I(u,v) \cup Z_u \cup Z_v) \setminus \{u,v\}$ are not weakly toll extreme vertices of $G$. Therefore, $ext(G) \subseteq T[u] \cup T[v]$.
\end{proof}

The following result is a consequence of Theorem~\ref{thm:wtn}~$(\ref{ite:2classes})$.

\begin{corollary} \label{cor:atmost2}
For any graph $G$, at most $2$ classes of twins of $G$ are formed by weakly toll extreme vertices.
\end{corollary}

Theorem~\ref{thm:wtn} leads to the following algorithm for finding a minimum weakly toll interval set of a general graph $G$.

\begin{itemize}

\item It begins determining the classes of twins $T_1, \ldots, T_k$ of $G$ that are toll extreme vertices. 
Corollary~$(\ref{cor:atmost2})$, we know that $k \le 2$.

\item If $k = 0$, then, by Theorem~\ref{thm:wtn}~$(\ref{ite:noextreme})$, $wtn(G) \le 8$. Then, a minimum weakly toll interval set of $G$ can be computed in $O(n^8 \alpha)$, where $\alpha$ is the time complexity for computing the weakly toll interval of a given set.

\item If $k = 1$, then Theorem~\ref{thm:wtn}~$(\ref{ite:1extreme})$ guarantees that $|T[u]| + 1 \le wtn(G) \le |T[u]| + 5$ where $T[u]$ is the only class of twins of $G$ formed by weakly toll extreme vertices. Then, a minimum weakly toll interval set of $G$ can be computed in $O(n^5 \alpha)$.

\item If $k = 2$, then, using Theorem~\ref{thm:wtn}~$(\ref{ite:2extremes})$, we have that $|T[u] \cup T[v]| \le wtn(G) \le |T[u] \cup T[v]| + 2$, where $T[u]$ and $T[v]$ are the only classes of twins of $G$ formed by weakly toll extreme vertices. In this case, a minimum weakly toll interval set of $G$ can be computed in $O(n^2 \alpha)$.

\end{itemize}

\begin{corollary} \label{cor:wtn}
Given a general graph $G$ of order $n$, size $m$ and maximum degree $\Delta$, a minimum weakly toll interval set $S$ of $G$ can be computed in $O(n^9\Delta^2 (n+m))$ steps. Furthermore, $|\hat{S}| \le 8$.
\end{corollary}

\begin{proof}
The time complexity, given in Corollary~\ref{cor:complexityIH}, for computing the weakly toll interval of a set $S$ is $O(\Delta^2 |S|^2 (n - |S|)(n+m))$. Note that $I(S) = S \cup I(\hat{S})$.
Note also that for any set $S$ considered in the above algorithm, $|\hat{S}| \le 8$, resulting in a total time complexity $O(n^9\Delta^2 (n+m))$.
\end{proof}

\section{Weakly toll hull number}

We remark that the polynomial-time algorithm presented in Section~\ref{sec:interval} for finding a minimum weakly toll interval set $S$ of a graph $G$ of order $n$ and size $m$ can be adapted for finding a minimum weakly toll hull set $R$ of $G$ in polynomial time. As observed in Corollary~\ref{cor:wtn}, $|\hat{S}| \le 8$. Since $wth(G') \le wtn(G')$ for any graph $G'$, we conclude that $|\hat{R}| \le 8$. Therefore, using Corollary~\ref{cor:complexityIH}, one can compute $wth(G)$ in $O(n^{12} \Delta^2 (n+m))$ steps. In this section, we present an algorithm with time complexity $O(n^2\Delta^2 (n+m))$ for computing ${\mathit wth}(G)$. The case where $G$ is a prime graph follows from the following known result and the fact that every induced path is a weakly toll walk.

\begin{lemma}{\em \cite{dourado-et-al}} \label{the:mhp}
	If $G$ is prime that is not complete, then every pair $u,v$ of non-adjacent vertices of $G$ is a monophonic hull set of $G$.
\end{lemma}

\begin{corollary}\label{cor:prime-not-clique-1}
	If $G$ is a prime graph that is not complete, then every pair of non-adjacent vertices of $G$ is a weakly toll hull set of $G$.
\end{corollary}

For the case where $G$ is not a prime graph, we need of some auxiliary results.

\begin{lemma} {\em \cite{leimer1993}}  \label{lem:clique}
	Let $M_i$ and $M_j$ be two mp-subgraphs of a graph $G$. If $V(M_i)\cap V(M_j)$ is not empty, then it is a clique.
\end{lemma}

\begin{lemma}\label{lem:prime-not-clique-2}
Let $G$ be a reducible graph and let $M$ be an extremal mp-subgraph of $G$. If $\exclusive{M}$ is not a clique, then there exist non-adjacent vertices $u_0,u_1\in \exclusive{M}$, a vertex $w_0\in \shared{M}$, and an induced path $P$ such that either $P=u_0 w_0 u_1$ or $P=u_0 w_0 w_1 u_1$ for some $w_1\in \shared{M}$.
\end{lemma}

\begin{proof}
	Let $S=\{u\in V(M_i)\setminus C_i : u \ \text{has at least one neighbor in} \ C_i\}$. If $(V(M_i)\setminus C_i)\setminus S=\emptyset$, then, since $V(M_i)\setminus C_i$ is not a clique, $S$ is not a clique.
	If $(V(M_i)\setminus C_i)\setminus S\neq\emptyset$, then $S$ separates some vertex $x\in (V(M_i)\setminus C_i)\setminus S$ from some vertex $y\in C_i$. Thus, $S$ is not a clique because $M_i$ contains no clique separator. From the above arguments, there exist nonadjacent vertices $u_0,u_1\in S$. If $u_0,u_1$ have a common neighbor $w_0\in C_i$, then the lemma follows. Otherwise, let $w_j\in C_i$ be a neighbor of $u_j$, for $j\in\{0,1\}$, and consider the induced path $P=u_0 w_0 w_1 u_1$.
\end{proof}

\begin{lemma} {\em \cite{leimer1993}} \label{lem:extremalmp}
	Every reducible graph has at least two extremal mp-sub\-graphs.
\end{lemma}

\begin{lemma} \label{lem:2MP}
Let $M_1$ and $M_2$ be distinct extremal mp-subgraphs of a graph $G$.
If $u_1 \in \exclusive{M_1}$ and $u_2 \in \exclusive{M_2}$, then $V(G) \setminus (\exclusive{M_1} \cup \exclusive{M_2}) \subset H(u_1,u_2)$.
\end{lemma}

\begin{proof}
Let $v \in V(G) \setminus (\exclusive{M_1} \cup \exclusive{M_2})$.
Let $P_1$ be a $(u_1,v)$-path containing only one vertex $v_1$ of $\shared{M_1}$ and containing the minimum number of vertices of $\shared{M_2}$.
Let $P_2$ be a $(u_2,v)$-path containing only one vertex $v_2$ of $\shared{M_2}$ and containing the minimum number of vertices of $\shared{M_1}$.
If $P_1$ has exactly one vertex of $\shared{M_2}$, we can assume that such vertex is $v_2$. Analogously,
if $P_2$ has exactly one vertex of $\shared{M_1}$, we can assume that such vertex is $v_1$.
Therefore, if $P_1$ has at most one vertex of $\shared{M_2}$ and $P_2$ has at most one vertex of $\shared{M_1}$, the concatenation of $P_1$ and $P_2$ is a weakly toll $(u_1,u_2)$-walk containing $v$.

Hence, since $\shared{M_2}$ is a clique, we can assume, without loss of generality, that $P_1$ has exactly two vertices of $v_2,v'_2$ of $\shared{M_2}$. Assume that $v_2$ is between $u_1$ and $v'_2$ in $P_1$. Let $M'_2$ be an extremal mp-subgraph of $G$ such that $\shared{M_2} \subseteq V(M'_2)$. By the choice of $P_1$, we have that $v'_2 \notin V(M_1)$, which means that $M'_2 \ne M_1$. 

We claim that $v_1$ does not belong to $V(M'_2)$. Then, suppose by contradiction that $v_1 \in V(M'_2)$.
Since $M'_2$ is an mp-subgraph, we know that $v_2$ does not separate $v_1$ from $v'_2$.
Then, there is a $(v_1,v'_2)$-path $P'_1$ in $M'_2 - v_2$. By the choice of $P_1$, we have that $P'_1$ has also two vertices of $\shared{M_2}$, $v'_2$ and a vertex $v''_2$.
Analogously, $\{v_2,v''_2\}$ does not separate $v_1$ from $v'_2$. However, this implies that $\shared{M_2}$ is an infinite set which is a contradiction, which means that the claim does hold.

Let $w_1$ be a vertex of $V(M'_2) \setminus \shared{M_2}$ that is neighbor of $v_2$ in $G$, and let $W$ be a closed $(w_1,w_1)$-walk of $M'_2 \setminus \shared{M_2}$ containing all vertices of $V(M'_2) \setminus \shared{M_2}$.
Using $W$, $P_1(u_1,v_2)$ and a $(u_2,v_2)$-path of $M_2$ containing only one vertex of $\shared{M_2}$, we conclude that $V(M'_2) \subset H(u_1,u_2)$.

Let $w_2$ be a neighbor of $v'_2$ in $G$ that belongs to $M'_2 - \shared{M_2}$.
Finally, the walk formed by $w_2v'_2$ and $P_2$ is a weakly toll containing $v$, completing the proof.
\end{proof}

\begin{theorem}\label{thm:wth}
	The following hold for a graph $G$ of order $n$.

	\begin{enumerate}[$(1)$]
		\item ${\mathit wth}(G)=n$ if $G$ is a complete graph. \label{ite:complete}
	
		\item ${\mathit wth}(G) = 2$ if
		
		\begin{enumerate}[$(a)$]
		\item If $G$ is a non-complete prime graph, or. \label{ite:prime}
		
		\item $G$ contains at least three extremal mp-subgraphs, or. \label{ite:3}		
		
		\item $G$ is reducible and $M_i$ is an extremal mp-subgraph of $G$ such that $\exclusive{M_i}$ is not a clique. \label{ite:notclique}
		\end{enumerate}

		\item $wth(G) \in \{2, x_1 + 1, x_2 + 1, x_1 + x_2\}$, where $x_i = |\exclusive{M_i}|$ for $i\in\{1,2\}$ if $G$ contains exactly two extremal mp-subgraphs $M_1$ and $M_2$. \label{ite:exactly2}		
	\end{enumerate}
\end{theorem}

\begin{proof}
\noindent $(\ref{ite:complete})$ In a complete graph, no vertex can be an internal vertex of a weakly toll walk. Thus, ${\mathit wth}(G)=n$.
	
\bigskip\noindent $(\ref{ite:prime})$ It follows from Corollary~\ref{cor:prime-not-clique-1}.

\bigskip\noindent $(\ref{ite:3})$
Let $M_1, M_2$ and $M_3$ be three extremal mp-subgraphs of $G$. Let $u_i \in \exclusive{M_i}$ for $i\in\{1,2,3\}$.
By Lemma~\ref{lem:2MP}, we have that $V(G) \setminus (\exclusive{M_i} \cup \exclusive{M_j}) \subset H(u_i,u_j)$ for any distinct $i,j \in \{1,2,3\}$.
Since $\exclusive{M_k} \subset V(G) \setminus (\exclusive{M_i} \cup \exclusive{M_j})$, for $\{i,j,k\} = \{1,2,3\}$, we have that $\{u_1,u_2\}$ is a weakly toll hull set of $G$.
	
\bigskip\noindent $(\ref{ite:notclique})$ By Lemma~\ref{lem:prime-not-clique-2}, there exist non-adjacent vertices $u_0,u_1 \in \exclusive{M_i}$, $w_0 \in \shared{M_i}$, and an induced path $P$ such that either $P=u_0w_0u_1$ or $P=u_0w_0w_1u_1$ for some $w_1\in \shared{M_i}$. By Corollary~\ref{cor:prime-not-clique-1}, $V(M_i)\subseteq H(u_0,u_1)$. Now, let $G'$ be the connected component of $G-V(M_i)$ containing $V(M_j)\setminus \shared{M_i}$, 	where $M_j \ne M_i$ is an mp-subgraph of $G$ containing $\shared{M_i}$. The definition of mp-subgraph implies that $w_0$ has a neighbor $v \in (V(M_j)\setminus \shared{M_i})\subseteq V(G')$. In addition, there exists a closed walk $W_v$ in $G'$ starting and ending at $v$ such that $W_v$ contains all vertices of $G'$. Hence, either $u_0w_0W_vw_0u_1$ or $u_0w_0W_vw_0w_1u_1$ is a weakly toll walk containing all vertices of $G'$, i.e., $V(G')\subseteq I(u_0,u_1) \subseteq H(u_0,u_1)$.
	
Next, let $G''$ be another connected component of $G-V(M_i)$ distinct from $G'$ containing vertices of an mp-subgraph $M_k$ of $G$ such that $C'_i=V(M_i)\cap V(M_k)\neq\emptyset$, and let $w\in C'_i$. Note that $w$ has neighbors $b_1\in \exclusive{M_i}$, $b_2 \in V(M_j) \setminus \shared{M_i}$ and $y\in (V(M_k)\setminus C'_i)\subseteq V(G'')$. Also, there exists a closed walk $W_y$ in $G''$ starting and ending at $y$ such that $W_y$ contains all vertices of $G''$, implying that $b_1wW_ywb_2$ is a weakly toll walk containing all vertices of $G''$, i.e., $V(G'')\subseteq I(b_1,b_2) \subseteq H(u_0,u_1)$ because we know that $b_1,b_2 \in H(u_0,u_1)$. In conclusion, $\{u_0,u_1\}$ is a weakly toll hull set of $G$ and ${\mathit wth}(G)=2$.

\bigskip\noindent $(\ref{ite:exactly2})$
By $(\ref{ite:notclique})$, we can assume that $\exclusive{M_1}$ and $\exclusive{M_2}$ are cliques.
For $i \in \{1,2\}$, if $M_i$ is a complete graph, choose $u_i$ as any vertex of $\exclusive{M_i}$; otherwise, choose $u_i$ as a vertex of $\exclusive{M_i}$ having a non-neighbor in $\shared{M_i}$.
By Lemma~\ref{lem:2MP}, $V(G) \setminus (\exclusive{M_1} \cup \exclusive{M_2}) \subset H(u_1,u_2)$.

Note that if $M_1$ is not a complete graph, i.e., there is a pair $u,v \in V(M_1)$ such that $uv \not\in E(G)$, then there is other pair $u',v' \in V(M_1)$ such that $u'v' \not\in E(G)$ and $\{u,v\} \cap \{u,v'\} = \varnothing$. Indeed this hold because if this was not the case, then $V(M_1) \setminus \{u,v\}$ would be a clique separating $u$ from $v$, which contradicts the assumption that $M_1$ is an mp-subgraph. Note that this fact implies that if $M_1$ is not a complete graph, then no vertex of $M_i$ is an extreme vertex.

We have four cases to consider. In the first one, $u_1$ and $u_2$ are extreme vertices of $G$. By the above parapraph, we have that $M_1$ and $M_2$ are complete sets. Therefore, for $i \in \{1,2\}$, it holds $T[u_i] = \exclusive{M_i}$. Since $u_1$ and $u_2$ are extreme vertices, then $\exclusive{M_1} \cup \exclusive{M_2}$ is a minimum weakly toll hull set of $G$ and has size $x_1 + x_2$.

In the second case, $u_1$ is an extreme vertex but $u_2$ is not.
Then, there are vertices $v_1,v_2 \in V(G)$ such that $u_2 \in I(v_1,v_2)$.
If $M_2$ is a complete graph, then $v_1, v_2 \not\in V(M_2)$.
Since $M_1$ is a complete graph, at most one of $v_1$ and $v_2$ belongs to $V(M_1)$. If one of $v_1$ and $v_2$ belongs to $V(M_1)$, then we can assume that such vertex is $u_1$. 
Since $V(G) \setminus (\exclusive{M_1} \cup \exclusive{M_2}) \subset H(u_1,u_2)$, we have that $\exclusive{M_2} \subset H(u_1,u_2)$. Since all vertices of $\exclusive{M_1}$ are true twins, we have that the minimum weakly toll hull set of $G$ has cardinality $|\exclusive{M_1}| + 1 = x_1 + 1$.
If $M_2$ is not a complete graph, then $u_2$ has a non-neighbor in $\shared{M_2}$. Since $\shared{M_2} \subset H(u_1,u_2)$, we have that $\exclusive{M_1} \cup \{u_2\}$ is a weakly toll hull set of $G$. It is minimum because the assumption that $u_1$ is an extreme vertex implies that all vertices of $\exclusive{M_1}$ are extreme vertices as well.
 
The third case, $u_2$ is an extreme vertex but $u_1$ is not, is analogue to the second case and we conclude that $wth(G) = x_2+1$.

In the last case, both $u_1$ and $u_2$ are not extreme vertices of $G$.
If $M_1$ is not a complete graph, then $u_1$ has a non-neighbor in $\shared{M_1}$.
Then, the fact that $V(G) \setminus (\exclusive{M_1} \cup \exclusive{M_2}) \subset H(u_1,u_2)$ implies that $\exclusive{M_1} \subset H(u_1,u_2)$. Since the same does hold for $M_2$, it remains to consider that at least one of $M_1$ and $M_2$ is a complete graph.
Then, consider that $M_1$ is a complete graph.
Recall that there are vertices $v_1,v_2 \in V(G)$ such that $u_1 \in I(v_1,v_2)$.
Since $M_1$ is a complete graph, $v_1,v_2 \notin V(M_1)$.
Since $\exclusive{M_2}$ is a clique, at most one of $v_1$ and $v_2$ belongs to $\exclusive{M_2}$.
If one of $v_1$ and $v_2$ belongs to $\exclusive{M_2}$, then we can assume that such vertex is $u_2$.
Since $V(G) \setminus (\exclusive{M_1} \cup \exclusive{M_2}) \subset H(u_1,u_2)$, we have that $\{u_1,u_2\}$ is a weakly toll hull set of $G$.	
\end{proof}

\begin{corollary}
Given a general graph $G$ of order $n$, size $m$ and maximum degree $\Delta$, ${\mathit wth}(G)$ can be computed in $O(n^2\Delta^2 (n+m))$ steps.
\end{corollary}

\begin{proof}
Items~$(\ref{ite:complete})$ and~$(\ref{ite:prime})$ of Theorem~\ref{thm:wth} consider the possibilites for a prime graph. Every graph that is not a prime graph falls in one of items~$(\ref{ite:notclique})$ to~$(\ref{ite:3})$ because every reducible graph has at least two extremal mp-subgraph by Lemma~\ref{lem:extremalmp}.
We recall that the set of mp-subgraphs of $G$ can be computed in $O(nm)$ steps~\cite{leimer1993}. Items~$(\ref{ite:notclique})$ and~$(\ref{ite:3})$ can be easily teste in $O(nm)$ steps.
In item~$(\ref{ite:exactly2})$, we have to test whether a vertex is extreme. Using Corollary~\ref{cor:complexityIH}, this can be done in~$O(n^2\Delta^2 (n+m))$ steps.	
\end{proof}

\section{Concluding remarks}

Theorem~\ref{thm:wtn}~$(\ref{ite:2classes})$ suggests that a set of weakly toll extreme vertices having no twins is an independent set. In the following result, we show that this fact holds on most of the well-studied graph convexities. Recall the following convexities:

\begin{itemize}

\item In the {\em $P_3$-convexity} a set $S$ is convex if for any two vertices $x,y \in S$, any common neighbor to $x$ and $y$ is also in $S$.

\item In the {\em induced $P_3$-convexity} a set $S$ is convex if for any two nonadjacent vertices $x,y \in S$, any common neighbor to $x$ and $y$ is also in $S$.

\item In the {\em triangle path convexity} a set $S$ is convex if for any two vertices $x,y \in S$, the vertices belonging to any $(x,y)$-path containing whose chords only form triangles also belong to $S$.

\item In the {\em $\Delta$-convexity} a set $S$ is convex if for any two adjacent vertices $x,y \in S$, any common neighbor to $x$ and $y$ is also in $S$.

\end{itemize}

\begin{proposition}
Let $G$ be a graph. In the geodesic, monophonic, $P_3$, induced $P_3$, triangle path, toll and weakly toll convexities, if $S \subseteq ext(G)$ such that $T(u) = \emptyset$ for every $u \in S$, then $S$ is an independent set.
\end{proposition}

\begin{proof}
Suppose by contradiction that $u,v \in S$ and $uv \in E(G)$. Without loss of generality, we can say that there is $u' \in N[u] \setminus N[v]$. Note that $u \in I(u',v)$, which contradicts the assumption that $u$ is an extreme vertex.
\end{proof}

Finally, we remark that the above result does not always hold in the $m^3$-convexity and  the $\Delta$-convexity.


\end{document}